\documentclass[A4paper]{article}

\usepackage[authoryear]{natbib}
\usepackage{subfig}
\usepackage[labelfont=bf]{caption}
\usepackage{graphicx}
\usepackage{multirow}%
\usepackage{tabularx,array}%
\usepackage{amsmath,amssymb,amsfonts,amsthm,stmaryrd,mathrsfs}%
\usepackage[title]{appendix}%
\usepackage{xcolor}%
\usepackage{rotating}
\usepackage{hyperref}
\hypersetup{
    colorlinks=true,
    linkcolor=blue,
    citecolor=cyan,      
    urlcolor=cyan
    }
\usepackage{booktabs}%
\usepackage{tikz}
	\usetikzlibrary{arrows.meta}
	\usetikzlibrary{shapes}
	\usetikzlibrary{calc}
	\usetikzlibrary{positioning,intersections}
	\usetikzlibrary{shapes.geometric}
	\usetikzlibrary{matrix} 
	\usetikzlibrary{fit}
\usepackage{tikz-qtree}
\usepackage{environ}
\usepackage{etoolbox}
\makeatletter
 \def\@makefnmark{\hbox{\@textsuperscript{\normalfont\@thefnmark}}}

\usepackage{environ}
\makeatletter
\newsavebox{\measure@tikzpicture}
\NewEnviron{scaletikzpicturetowidth}[1]{%
  \def\tikz@width{#1}%
  \begin{lrbox}{\measure@tikzpicture}%
  \BODY
  \end{lrbox}%
  \pgfmathparse{#1/\wd\measure@tikzpicture}%
  \BODY
}
    
\makeatother
\allowdisplaybreaks

\theoremstyle{thmstyleone}%
\newtheorem{theorem}{Theorem}%  meant for continuous numbers
\newtheorem{proposition}[theorem]{Proposition}% 

\theoremstyle{thmstyletwo}%
\newtheorem{remark}{Remark}%

\theoremstyle{thmstylethree}%
\newtheorem{definition}{Definition}%

\begin{document}

\title{Kleisli semantics and hypergraph composition for Greimasian narrative programs}
\maketitle

\author{Michael Fowler: synthifou@gmail.com}

\begin{abstract}
This article proposes a category-theoretic formalization of Greimasian narrative programs (NPs) that makes their compositional structure mathematically precise. Building on a reconstruction of the actantial model as a categorical schema, we introduce a refined typological schema of actants and derive Set-valued instances corresponding to role-indexed elements of a narrative. NPs are represented within a categorical schema whose morphisms are interpreted using monads on $\mathbf{Set}$. In particular, the List monad provides a Kleisli semantics for modeling non-atomic, list-valued actantial configurations, while the Maybe monad encodes optional dependencies between programs. This yields a minimal representation of narrative programs as structured data with an intrinsic compositional interpretation. To account for the dynamics of narrative formation, we lift these constructions into a diagrammatic setting by freely generating a symmetric monoidal category, and subsequently a hypergraph category, from the set of actants. In this framework, narrative programs act as generators of morphisms, and their composition is realized through wiring diagrams. A narrative trajectory is thereby interpreted as a single composite morphism. This approach provides a unified mathematical framework for structural semiotics, connecting data-level representations of narrative elements with their compositional realization in discourse.
\end{abstract}

%\keywords{conceptual graphs, semantic graphs, Fran\c{c}ois Rastier, structural semiotics}



\section{Introduction}

By the late 1960s, the research program of the semiotician Algirdas Julien Greimas had begun to consolidate into what is now known as the Paris School of structural semiotics. Drawing on Ferdinand de Saussure's structural linguistics of difference, Roland Barthes's semiology of narrative, and Claude L\'{e}vi-Strauss's analysis of constitutive units through homologation, this tradition developed a systematic framework for the analysis of meaning in both literary and visual discourses. \citep{GreimasCollinsPerron1989} Central to this framework is the idea that meaning emerges not from isolated elements, but from structured relations between abstract roles and their instantiations.

Greimas built a framework for a structural semiotics that was based on the notion of a \textit{seme}, or smallest irreducible component of meaning, that in combination form a \textit{sememe}, or the smallest compound unit of meaning. A seme is analogous to a morpheme or lexeme in linguistics or a phoneme in phonology. A seme is not autonomous or `atomistic,' but instead ``exists only because of the differential gap that opposes it to other semes.'' \citep[page 278]{GreimasCourtes1979} That is to say that semes themselves embody a structure of opposition, or what Greimas identified as \textit{semic categories}. \citep[page 16]{Greimas1987}  

These semic categories constitute the \textit{content plane}, and represent the semantic or meaningful layer of signification, and thus organise  concepts, values, and narrative structures. \citep[page xix]{Greimas1983} This plane is not tied to any specific mode of expression---such as a language, image, or gesture---but is instead the deep, generative level from which meaning emerges across different forms. This sits in contrast to the \textit{expression plane} which is how the content becomes manifested. In the example of a literary text, the expression plane includes the syntax, vocabulary, and narrative style it uses, whereas the content plane contains the underlying narrative functions, actants (roles), and thematic structures. Here, Greimas emphasized that the relationship between these planes is not fixed but dynamic. It is shaped by semiotic processes such as modalisation and isotopy which mediate between deep structures and surface-level manifestations.

Among Greimas's most influential contributions is the formalization of narrative programs (NPs) which provide a minimal schema for describing transformations in a narrative. Here, Greimas was concerned with identifying elementary forms of narrativity, firstly between a `Subject' and an `Object,' and how utterances describe their states, transformations and relations. \citet[page 90]{Greimas1987} proposed that there exists a functional relation between Subject and Object as a semic category in which junctions (links) can be described either through:
\begin{align}
\textit{Conjunctive utterances}\,&=S \cap O,\\
\textit{Disjunctive utterances}\,&=S \cup O.
\end{align}
Greimas uses the set theoretic symbols for union and intersection to denote the type of link between actants, while others such as \citet[page 160]{Hebert2020} use roman lowercase letters `u' (for $\cup$ as a disjunction) and `n' (for $\cap$ as a conjunction). We note here that the set theoretic symbols that Greimas uses should not be assumed to also imply their logical operator counterparts of `$\wedge$' as logical conjunction and `$\vee$' as logical disjunction. This is because the case of a disjunctive utterance written as 
\[
(\mathrm{Subject} \vee \mathrm{Object}),
\]
would yield a truth statement in which when $S$ is true and $O$ is true then both are true. But in Greimasian semiotics, disjunctive narrative programs are usually framed as \textit{virtualisations}, that is, when a Subject is `not in possession' of an Object, or an Object is `missing.' This would imply an exclusive disjunction ($\oplus$) as more appropriate: i.e. either the Subject exists or the Object exists but only one can be true. In this article we use the original notation of Greimas, though interpret the underlying semantics of `$\cup$' to represent an exclusive disjunction.

As H\'{e}bert notes, ``the formula for the narrative program . . . can be verbalized or explained as follows: the function by which a subject 1 (subject of doing) causes a subject 2 (subject of state) to be conjoined with (or disjoined from) an object (object of state).'' \citep[page 160]{Hebert2020} NP formulas can thus be denoted in the following way:
\begin{align}\label{NPcap}
NP_{i}:&=F\{S_{1} \rightarrow (S_{2} \cap O)\},\\
NP_{i}:&=F\{S_{1} \rightarrow (S_{2} \cup O)\},\label{NPcup}
\end{align}
where $NP$ is an $i$-indexed narrative program with a junction type (conjunctive $=\cup$, disjunctive $=\cap$), $O$ is an object, and $S_{1}$ is a \textit{subject of doing} and $S_{2}$ is a \textit{subject of state}. Greimas defines a subject of doing as the actant that performs actions to acquire or relinquish an object, for which their ability to perform this task is based on competence, which comprises modal states such as \textit{being-able-to-do}, \textit{knowing-how-to-do} or \textit{having-to-do}. \citep[page 25]{RinghamRingham2006} This is in contrast to the subject of state which is also known as a \textit{patient} in  a narrative programme. Here, a subject of state is an actant that is defined by its relationship to a particular state, rather than by its actions. It contrasts with the subject of doing (i.e., an agent), which acts to change that state. While the subject of doing is dynamic and performs actions, the subject of state represents the modality of \textit{being} or a moment of stability (or lack thereof).

As found in the NPs given in Equation~(\ref{NPcap}) and Equation~(\ref{NPcup}), narrative programs admit a recursive structure through the inclusion of the function $F$, which allows for the nesting of transformations. For example, consider the followings NPs:
\begin{align}
NP_{1}:&=\{S_{1} \rightarrow (S_{2} \cap O_{1})\},\\
NP_{2}:&=\left\{NP_{1}\{S_{3} \rightarrow (S_{4} \cap O_{1})\}\right\}.
\end{align}
\citet[page 161]{Hebert2020} describes the example of $NP_{2}$ as a program of \textit{manipulation} that is instantiated through the modality of a \textit{causing-to-do}. The conjunction in $NP_{1}$ of the subject ($S_{1}$) with an object ($O_{1}$) is the cause for another subject ($S_{3}$) to acquire the same object ($O_{1}$). This nesting feature allows for the articulation of complex narrative dynamics from simple combinatorial forms. 

Despite their formal appearance, however, narrative programs have largely remained at the level of semiotic notation, without a fully explicit mathematical semantics. This paper proposes a categorical formalization of narrative programs that makes their compositional structure precise. Building on previous work in which the Greimasian actantial model was interpreted as a categorical schema, we extend this perspective by introducing Kleisli instances over the List monad. This allows us to relax the atomicity of database entries and represent narrative relations as list-valued structures, capturing the multiplicity of objects and actants involved in a given program. In this setting, narrative programs are represented minimally as list-valued bindings between actants and objects, together with a mode (conjunctive or disjunctive) encoded at the schema level.

To move from static representation to dynamic structure, we then reinterpret these constructions within the framework of wiring diagrams. By extracting a discrete category of role-indexed actants from a Set-valued instance, and freely generating a symmetric monoidal (and subsequently hypergraph) category, we obtain a diagrammatic language in which narrative programs act as generators of morphisms. This provides a compositional account of discoursivization, in which complex narratives arise through the substitution and composition of simpler programmatic units.

\subsection{Summary}

Starting from the categorical schema of the Greimas actantial model $\mathcal{A}$, we construct a refined typological schema of actants through:
\begin{equation}
\mathcal{A}\rightarrow\mathcal{A}_{D}^{\ast}\xrightarrow{\Delta} \mathcal{A}^{\prime}\xrightarrow{I}\mathbf{Set},
\end{equation}
which is an application of a constant functor to obtain a single-object category $\mathcal{A}^{\prime}$ whose Set-valued instances yield role-indexed actants. Narrative programs (NPs) are then represented within a categorical schema $\mathcal{N}$ whose morphisms are interpreted using monads on $\mathbf{Set}$. In particular, the List monad provides a Kleisli semantics for the aspect (morphism) `actorializes', allowing non-atomic, list-valued representations of actantial configurations, while the Maybe monad encodes the optional dependency relation between NPs. This yields a minimal representation of narrative programs as list-valued structures equipped with compositional semantics.

To model the dynamics of narrative composition, we pass from these data-level representations to a diagrammatic setting by constructing the free symmetric monoidal category $\mathbf{FSM}(\mathbf{Disc}(X))$ on the discrete category of actants. After introducing a narrative program generator $\llbracket p \rrbracket$ as a morphism, we equip this structure with Frobenius operations to obtain a hypergraph category $\mathcal{H}(X)$. In this setting, narrative programs act as generators of morphisms, and their composition is realized through wiring diagrams.

The main result of the paper establishes that a narrative trajectory is given by a morphism
\[
\nu : A \to B
\]
in $\mathcal{H}(X)$, constructed from narrative program generators via composition, tensor product, symmetry, and Frobenius structure. This provides a formal account of discoursivization as a compositional process, in which complex narratives arise from the structured interaction of simpler programmatic units.

\section{Formalising narrative programs}

In previous work we introduced a category-theoretic framework for understanding the role of actants within Greimas's \textit{actantial model}, which we drawn on in this article for the basis of our exploration of NPs.  The actantial model of Greimas was originally developed as a refinement of the work of \citet{Propp1984} in Russian fairytales, and sought to describe a relational system in which roles within narratives are defined not by intrinsic properties, but by their functional positions within a narrative configuration. Greimas originally depicted the model as follows:
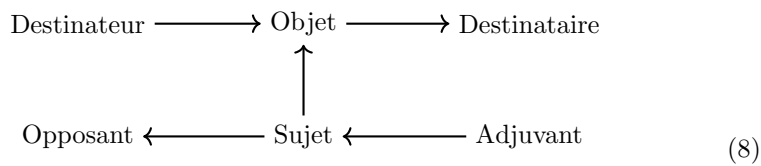
\begin{align}
\begin{tikzpicture}\label{GreimasAM}
\node[] at (0,0)(Objet){Objet};
\node[] at (-3,0)(Destinateur){Destinateur};
\node[] at (3,0)(Destinataire){Destinataire};
\node[] at (0,-1.5)(Sujet){Sujet};
\node[] at (3,-1.5)(Adjuvant){Adjuvant};
\node[] at (-3,-1.5)(Opposant){Opposant};
\draw[thick,->](Destinateur)--(Objet);
\draw[thick,->](Objet)--(Destinataire);
\draw[thick,->](Adjuvant)--(Sujet);
\draw[thick,->](Sujet)--(Opposant);
\draw[thick,->](Sujet)--(Objet);
\end{tikzpicture}
\end{align}
The model consists of a collection of actants and directed edges (indicating flow of influence). The actants are labelled: `Subject,' (\textit{Sujet}) `Object,' (\textit{Objet}) `Sender,' (\textit{Destinateur}) `Helper,' (\textit{Adjuvant}) `Opponent,' (\textit{Opposant}) and `Receiver' (\textit{Destinataire}) that are organised along semantic `axes' corresponding to the themes of `desire' (\{Subject, Object\}), `knowledge' (\{Sender, Object, Receiver\}) and `power' (\{Helper, Subject, Opponent\}). 

Actants are functional or \textit{syntagmatic units} that subsume narrative roles within a given trajectory rather than denoting fixed entities. \citep[p. 5]{GreimasCourtes1979} An actant does not correspond to a specific object or agent, but to an abstract role that may be instantiated by diverse kinds of entities depending on the narrative context. This further allows for \textit{actantial syncretism}, or a situation in which a single actor embodies multiple actants within a story, or an actant switches roles within a story. Actants themselves can be varied and are not limited to sentient beings but can include: forces of nature, agents, objects, collectives, or purely abstract/conceptual elements. \citep[p. 138]{Hebert2020} Actants are thus defined by the relations they sustain within a narrative configuration, and not by any intrinsic or ontological properties they posses. While the model is typically presented diagrammatically, its underlying structure admits a formal interpretation in terms of objects and relations, making it amenable to categorical reconstruction.

\subsection{Categorical schemas and the actantial model}

Given the original visualisation of the actantial model of Greimas we gave in Equation~(\ref{GreimasAM}), it was showed in \citet{Fowler2026} how the  model can be re-constructed under the framework of \textit{ontological logs} of \citet{SpivakKent2012}, thus establishing an equivalence \citep{Spivak2010} to a small category in $\mathbf{Set}$, which further allows for the analysis of actantial roles in a given narrative as instances on a codomain of the functor $I$. Below we give definitions of these structures.

\begin{definition}[Categorical schema]\label{CategoricalSchemaDef}\normalfont
A \emph{categorical schema} is a 2-tuple $\mathcal{C}=(G, \simeq)$ in which $G$ is a directed multi-graph, and $\simeq$ is a categorical path equivalence on $G$. An \emph{ontological log} (or \emph{olog}) of a categorical schema is the labelling of the vertices $V$, and arrows $A$ of $G$ such that:
\begin{itemize}\setlength\itemsep{0.1em}
\item a vertex $v$ represents a \emph{type} or set of objects that form a class, for which its label is written as a \emph{singular indefinite noun phrase} inscribed in a rectangle.
\item an arrow $a$ represents an \emph{aspect} or function between types, for which its label is written as a verb phrase.
\item an olog is accompanied with a set of equations that are path equivalences statements.
\end{itemize}
\end{definition}

\begin{definition}[Instance on a schema]\label{InstanceDef}\normalfont
Let $\mathcal{C}$ be a categorical schema. An instance on $\mathcal{C}$ is a functor
\begin{equation}
I : \mathcal{C} \to \mathbf{Set}
\end{equation}
such that objects of $\mathcal{C}$ are mapped to sets of instances and morphisms are mapped to functions respecting all specified path equivalences.
\end{definition}

%FIG....................................................................................................................
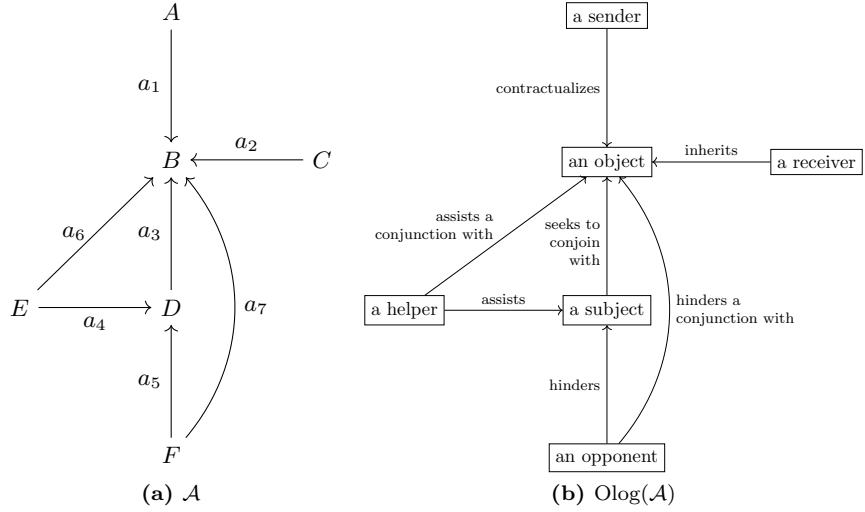
\begin{figure}
\centering
\subfloat[$\mathcal{A}$\label{CategoryA}]{
\begin{tikzpicture}[node distance=15mm,font=\small]
\node at (0,0)(S1){$A$};
\node[below=of S1] (O1){$B$};
\node[right=of O1] (R){$C$};
\node[below=of O1] (S2){$D$};
\node[left=of S2] (H){$E$};
\node[below=of S2] (O2){$F$};
\draw[->](S1)--node[left,pos=0.5]{$a_{1}$}(O1);
\draw[->](R)--node[above,pos=0.5]{$a_{2}$}(O1);
\draw[->](S2)--node[left,pos=0.5]{$a_{3}$}(O1);
\draw[->](H)--node[below,pos=0.5]{$a_{4}$}(S2);
\draw[->](O2)--node[left,pos=0.5]{$a_{5}$}(S2);
\draw[->](H)--node[left,pos=0.5]{$a_{6}$}(O1);
\draw[->](O2) to [bend right=40]node[right,pos=0.5]{$a_{7}$}(O1);
\end{tikzpicture}
}
%\hfill
\subfloat[$\mathrm{Olog}(\mathcal{A})$\label{OlogA}]{
\resizebox{0.56\textwidth}{!}{
\begin{tikzpicture}[A-node/.style={draw,rectangle},node distance=22mm]
\node[A-node] at (0,0)(S1){a sender};
\node[below=of S1,A-node] (O1){an object};
\node[right=of O1,A-node] (R){a receiver};
\node[below=of O1,A-node] (S2){a subject};
\node[left=of S2,A-node] (H){a helper};
\node[below=of S2,A-node] (O2){an opponent};
\draw[->](S1)--node[left,align=right,pos=0.5,font=\footnotesize]{contractualizes}(O1);
\draw[->](R)--node[above,align=right,pos=0.5,font=\footnotesize]{inherits}(O1);
\draw[->](S2)--node[left,align=right,pos=0.45,font=\footnotesize]{seeks to\\conjoin\\with}(O1);
\draw[->](H)--node[above,align=right,pos=0.5,font=\footnotesize]{assists}(S2);
\draw[->](O2)--node[left,align=right,pos=0.5,font=\footnotesize]{hinders}(S2);
\draw[->](H)--node[align=right,left,pos=0.5,font=\footnotesize,yshift=5,xshift=-3]{assists a\\conjunction with}(O1);%actualizes
\draw[->](O2) to [bend right=40]node[align=left,right,pos=0.5,font=\footnotesize]{hinders a\\conjunction with}(O1);%virtualizes
\end{tikzpicture}
}
}
\caption{\textbf{(a)} Categorical schema $\mathcal{A}$ of the Greimas actantial model with path equivalences: $a_{3}\circ a_{4}\simeq a_{6}$, and $a_{3}\circ a_{5}\simeq a_{7}$. \textbf{(b)} Ontological log of the categorical schema $\mathcal{A}$ of the Greimas actantial model with equivalent aspects given as: $(\textrm{seeks to cojoin with})\circ (\textrm{assists})\simeq (\textrm{assists a conjunction with})$, and $(\textrm{seeks to cojoin with})\circ (\textrm{hinders})\simeq (\textrm{hinders a conjunction with})$.}
\label{ActantialModel}
\end{figure}
%.....................................................................................................................

In Figure~\ref{ActantialModel} we give both a diagram of the small category $\mathcal{A}$ of the Greimas actantial model, and its accompanying olog (Figure~\ref{OlogA}). In order to motivate our exploration of NP formulation, we turn to the narrative of \textit{The Hare \& the Tortoise} for use as an example case study (see Appendix~\ref{Hare+Tortoise}). Using the actantial model as a basis for establishing the different actor roles found within the Aesop fable, we apply $I(x)$ for all $x \in \mathcal{A}$ to derive the following collection of actants:
\begin{align*}
I(\mathrm{Subject})&= \{\mathrm{Tortoise, Hare}\},\\
I(\mathrm{Object})&= \{\mathrm{race \,\,win, challenge, justification}\},\\
I(\mathrm{Sender})&= \{\mathrm{Tortoise}\},\\
I(\mathrm{Receiver})&= \{\mathrm{Tortoise}\},\\
I(\mathrm{ Helper})&= \{\mathrm{consistency, perseverence, Fox, fairness}\},\\
I(\mathrm{Opponent})&= \{\mathrm{overconfidence, underestimation, laziness, nap}\}.
\end{align*}
As \citet[page 9]{Spivak2010} suggests, because each categorical scheme consists of a multi-graph $G$ together with an equivalence relation, each object $x \in \mathrm{Ob}(\mathcal{A})$ can be represented as a table with an `ID' column and instances as rows. Each arrow from an object $x$ to another object $x^{\prime}$ in $\mathcal{A}$ is represented in the table of $x$ as a foreign key---that is, a column whose label points to another table, and whose row data is in the ID column of the corresponding foreign key. For example consider the following two linked tables from $\mathrm{Olog}(\mathcal{A})$ in Figure~\ref{OlogA}:
\begin{equation}
\begin{tikzpicture}]font=\small]
\tikzset{
  ObjectTable/.pic={
		\node at (0,0)(Object){
			\begin{tabular}{@{} | l || @{}}
			\hline
			\multicolumn{1}{|c|}{\textbf{an object}}\\ \hline
			\textbf{ID}\\ \hline
			race win\\ \hline
			challenge\\ \hline
			justification\\ \hline
			\end{tabular}
			};
  },
  SenderTable/.pic={
 		\node at (0,0)(Sender){
			\begin{tabular}{| l || l |}
			\hline
			\multicolumn{2}{|c|}{\textbf{a sender}}\\ \hline
			\textbf{ID}&\textbf{contractualises}\\ \hline
			Tortoise&challenge\\ \hline
			\end{tabular}
			};
  		}
}
\pic [anchor=north,local bounding box=Table1] at (0,0) {SenderTable}; 
\pic [local bounding box=Table2,right=40mm of Table1.north,anchor=north] {ObjectTable}; 
\end{tikzpicture}
\end{equation}

In this setup though, we are limited to \textit{atomic} database instances. By this we mean that every instance as $I(x) \in \mathcal{A}$ is represented through a separate row. In the next section we follow Spivak's innovative use of Kleisli $\top$-instances \citep{Spivak2012} in order to relax the atomicity requirement for databases. This will allow us to represent NPs that follow the form
\[
\{\mathrm{Hare} \rightarrow (\textrm{Hare} \cup \{\textrm{consistency, perseverance}\})\},
\]
through database entries that associate an actant to a set of actants or vice versa. This approach not only allows for more compact representations of actantial data for syncretic or typological modelling, but aligns to what \citet[page 278]{GreimasCourtes1979} identify as the fundamental property of structural semiotics---that semic categories on the content plane, and their subsequent compound constructs such as NPs, isotopies and molecules are \textit{relational} and never purely \textit{substantial}. 

\subsection{Kleisli database instances}

Although monads and their Kleisli categories have been commonly applied in computer science through parsers and language processing, they can provide a structural approach to semiotics with a unique utility. We will use them to firstly store NPs in list form, and then use them as the basis to further explore the position that as ``simple syntactic units,'' \citep[page 206]{GreimasCourtes1979} NPs extend and compose \textit{narrative trajectories} across myriad instances. 

In order to represent set-valued elements from the Aesop fable that account for both conjunctive (Equation~(\ref{NPcap})) and disjunctive (Equation~(\ref{NPcup})) NPs, we begin with a refining of the categorical schema of the Greimas actantial model through deriving a category $\mathcal{A}_{D}^{\ast}$, which is then collapsed by a \textit{constant functor} $\Delta: \mathcal{A}_{D}^{\ast}\rightarrow\mathcal{A}^{\prime}$, after which a Set-valued instance $I:\mathcal{A}^{\prime}\rightarrow \mathbf{Set}$ yields the underlying role-indexed actants. We sumerise this construction as:
\begin{equation}
\mathcal{A}\rightarrow\mathcal{A}_{D}^{\ast}\xrightarrow{\Delta} \mathcal{A}^{\prime}\xrightarrow{I}\mathbf{Set}.
\end{equation}

\begin{definition}[Refined typological schema of actants]\normalfont
Let $\mathcal{A}$ be the categorical schema corresponding to the Greimas actantial model (Figure~\ref{CategoryA}). We define its \emph{refined typological schema of actants} to be the discrete category
\begin{equation}
\mathcal{A}_{D}^{\ast} := \mathbf{Disc}(R),
\end{equation}
where $R := (\mathrm{Ob}(\mathcal{A}) \setminus \{D\}) \cup \{S_1,S_2\}$, with $D$ denoting the type labelled `a subject' in $\mathrm{Ob}(\mathcal{A})$.
We introduce two new objects:
\begin{align}
&S_{1} \quad \text{(`a subject of doing')},\\
&S_{2} \quad \text{(`a subject of state')}.
\end{align}
such that the morphisms of $\mathcal{A}_{D}^{\ast}$ are only identities:
\begin{align}
\mathrm{Hom}_{\mathcal{A}_{D}^{\ast}}(u,u)&=\{\mathrm{id}_{u}\},\\
\mathrm{Hom}_{\mathcal{A}_{D}^{\ast}}(u,v)&=\varnothing, \quad \text{for } u \neq v.
\end{align}
\end{definition}

\begin{remark}
The construction given above retains the typological distinctions of the actantial model while refining the subject position to reflect the argument structure of narrative programs. A replacement of $D \in \mathrm{Ob}(\mathcal{A})$ by $S_{1}$ and $S_{2}$ is necessary because the original actantial schema does not distinguish between the \textit{subject of doing} and the \textit{subject of state}, whereas this distinction is fundamental to the internal structure of narrative programs. By introducing two distinct subject-types prior to collapse via the constant functor, we ensure that List-valued Kleisli representations retain the full argument structure of Greimas's original formulation.
\end{remark}

\begin{definition}[Actantial instance]\label{ActantialInstanceDef}\normalfont
Let $\mathcal{A}_{D}^{\ast}$ be the refined typological schema of actants and let $\Delta: \mathcal{A}_{D}^{\ast}\rightarrow \mathcal{A}^{\prime}$ be the \emph{constant functor} for which $\Delta_{a}(x)=a$ for every object $x \in \mathcal{A}_{D}^{\ast}$, and $\Delta_{a}(f)=\mathrm{id}_{a}$ for every morphism $f: x \rightarrow x^{\prime}$ in $\mathcal{A}_{D}^{\ast}$. We label the unique object $a$ in $\mathcal{A}^{\prime}$ as the type `an actant.' An \emph{actantial instance} relative to the refined typological reduction is a functor 
\[
I:\mathcal{A}^{\prime}\rightarrow \mathbf{Set},
\]
for which we write $X:=I(a)$ for the underlying set of role-indexed actants. Its elements are denoted $x_{r} \in X$, where $x$ is a discourse entity and $r$ indicates its actantial role.
\end{definition}

Using the actantial instance construction allows us to populate a database table that associates set-valued instances to the single object $a \in \mathcal{A}^{\prime}$ for the Aesop fable \textit{The Hare \& the Tortoise}. This yields:
\begin{equation}
\begin{tikzpicture}
\node at (0,0){
\begin{tabular}{@{} | l || @{}}
\hline
\multicolumn{1}{|c|}{\textbf{an actant}}\\ \hline
\textbf{ID}\\ \hline
race win$_{Object}$\\ \hline
challenge$_{Object}$\\ \hline
justification$_{Object}$\\ \hline
Tortoise$_{SubjectDoing}$\\ \hline
Tortoise$_{SubjectState}$\\ \hline
Hare$_{SubjectDoing}$\\ \hline
Hare$_{SubjectState}$\\ \hline
Tortoise$_{Sender}$\\ \hline
Tortoise$_{Receiver}$\\ \hline
consistency$_{Helper}$\\ \hline
perserverence$_{Helper}$\\ \hline
Fox$_{SubjectDoing}$\\ \hline
Fox$_{SubjectState}$\\ \hline
fairness$_{Helper}$\\ \hline
overconfidence$_{Opponent}$\\ \hline
underestimation$_{Opponent}$\\ \hline
laziness$_{Opponent}$\\ \hline
nap$_{Opponent}$\\ \hline
\end{tabular}
};
\end{tikzpicture}
\end{equation}
Given the schema $\mathcal{A}$ is a category that establishes relations between types through an \textit{actantial grammar} \citep[page 5]{GreimasCourtes1979} on roles, we can use its collapse in the object $a \in \mathrm{Ob}(\mathcal{A}^{\prime})$ as the basis for a new schema $\mathcal{N}$ that accounts for NP construction. In particular, we want to track the sign of the junction of an NP (i.e., $\cap$ or $\cup$) as well as the singleton, or list-formatted actants that associate to $I(\mathcal{A}^{\prime})$, which in turn form syntagmatic units that generate larger hypotactical clusters at the meso- and macro-scale of narrative. We firstly provide definitions after \citet{Spivak2012} of a \textit{monad}, \textit{Kleisli category} and $\top$-\textit{instance}, as these objects serve the underlying structure of our proposed schema. 

\begin{definition}[Monad]\label{MonadDef}\normalfont%\citep[page 6]{Spivak2012}
A \emph{monad}, denoted $\top$ on \textbf{Set} is the triple $\top:=(T, \eta, \mu)$ such that $T:\mathbf{Set}\rightarrow \mathbf{Set}$ is a functor, and $\eta: \textrm{id}_{\mathbf{Set}}\rightarrow T$ (the \emph{unit map}) and $\mu: T \circ T \rightarrow T$ (the \emph{multiplication map}) are natural transformation such that the following diagrams commute:
\begin{align}
\begin{tikzpicture}
\node at (0,0)(Tcirc){$T \circ \textrm{id}_{\mathbf{Set}}$};
\node at (3,0)(TT){$T \circ T$};
\node at (3,-2)(T){$T$};
\draw[->](Tcirc)--node[above,pos=0.5]{$\textrm{id}_{T}\circ \eta$}(TT);
\draw[->](TT)--node[right,pos=0.5]{$\mu$}(T);
\draw[double](T)--(Tcirc);
\end{tikzpicture}\\
\begin{tikzpicture}
\node at (0,0)(Tcirc){$\textrm{id}_{\mathbf{Set}\circ T}$};
\node at (3,0)(TT){$T \circ T$};
\node at (3,-2)(T){$T$};
\draw[->](Tcirc)--node[above,pos=0.5]{$\eta \circ \textrm{id}_{T}$}(TT);
\draw[->](TT)--node[right,pos=0.5]{$\mu$}(T);
\draw[double](T)--(Tcirc);
\end{tikzpicture}\\
\begin{tikzpicture}
\node at (0,0)(Tcirc){$T \circ T \circ T$};
\node at (3,0)(TT){$T \circ T$};
\node at (0,-2)(TafterT){$T \circ T$};
\node at (3,-2)(T){$T$};
\draw[->](Tcirc)--node[above,pos=0.5]{$\mu \circ \textrm{id}_{T}$}(TT);
\draw[->](TT)--node[right,pos=0.5]{$\mu$}(T);
\draw[->](TafterT)--node[below,pos=0.5]{$\mu$}(T);
\draw[->](Tcirc)--node[left,pos=0.5]{$\textrm{id}_{T}\circ \mu$}(TafterT);
\end{tikzpicture}
\end{align}
\end{definition}

\begin{definition}[Kleisli category]\label{KleisliCategoryDef}\normalfont
Let $\top = (T, \eta, \mu)$ be a monad on \textrm{Set} such that there exists a \emph{Kleisli category} associated to the monad, which we denote $\mathbf{Kls}(\top)$, and whose objects are the sets $\textrm{Ob}(\mathbf{Kls}(\top))=\textrm{Ob}(\mathbf{Set})$, and 
\begin{equation}
\mathrm{Hom}_{\mathbf{Kls}(\top)}(X,Y)=\mathrm{Hom}_{\mathbf{Set}}(X,T(Y)).
\end{equation}
Given $f:X \rightarrow T(Y)$ and $g:Y\rightarrow T(Z)$, their Kleisli composite is
\begin{equation}
g \circ f:=\mu_{Z} \circ T(g)\circ f:X\rightarrow T(Z).
\end{equation}
\end{definition}

\begin{definition}[Kleisli $\top$-instance]\label{T-instanceDef}\normalfont
Let $\mathcal{C}$ be a categorical schema and let $\top:=(T,\eta,\mu)$ be a monad on $\mathbf{Set}$ with $\mathbf{Kls}(\top)$ its associated Kleisli category. A \emph{Kleisli} $\top$-\textit{instance} on $\mathcal{C}$ is a functor in the form
\begin{equation}
\delta: \mathcal{C} \rightarrow \mathbf{Kls}(\top).
\end{equation}
\end{definition}

%FIG....................................................................................................................
\begin{figure}
\centering
\subfloat[Table of instances for $P \in \mathrm{Ob}(\mathcal{N})$\label{TableN}]{
\resizebox{\textwidth}{!}{
\begin{tikzpicture}
		\node at (0,0)(Object){
			\renewcommand{\arraystretch}{1.25}
			\begin{tabular}{@{} | l || >{\raggedright\arraybackslash}p{1.9cm} | p{1.5cm} | p{1.7cm} | >{\raggedright\arraybackslash}p{5cm} |}
			\hline
			\multicolumn{5}{|c|}{\textbf{a narrative program}}\\ \hline
			\textbf{ID}&\textbf{is discoursivization of}&\textbf{has junction type}&$\epsilon$(\textbf{depends on})&$\mathbf\delta(\textbf{actorializes})$\\ \hline
			NP1&Aesop226&$\cap$&$\ast$&[Tortoise$_{SubjectDoing}$, Tortoise$_{SubjectState}$, challenge$_{Object}$]\\ \hline
			NP2&Aesop226&$\cap$&$\ast$&[Hare$_{SubjectDoing}$, Hare$_{SubjectState}$, underestimation$_{Opponent}$]\\ \hline
			NP3&Aesop226&$\cap$&$\ast$&[Fox$_{SubjectDoing}$, Fox$_{SubjectState}$, fairness$_{Helper}$]\\ \hline
			NP4&Aesop226&$\cap$&NP2&[Hare$_{SubjectDoing}$, Hare$_{SubjectState}$, nap$_{Opponent}$, laziness$_{Opponent}$, overconfidence$_{Opponent}$]\\ \hline
			NP5&Aesop226&$\cap$&NP7&[Tortoise$_{SubjectDoing}$, Tortoise$_{SubjectState}$, race win$_{Object}$, justification$_{Object}$]\\ \hline
			NP6&Aesop226&$\cup$&$\ast$&[Hare$_{SubjectDoing}$, Hare$_{SubjectState}$, race win$_{Object}$]\\ \hline
			NP7&Aesop226&$\cup$&$\ast$&[Hare$_{SubjectDoing}$, Hare$_{SubjectState}$, consistency$_{Helper}$, perserverence$_{Helper}$]\\ \hline
			\end{tabular}
			};
\end{tikzpicture}
}
}

\subfloat[$\mathrm{Olog}(\mathcal{N})$\label{OlogN}]{
\begin{tikzpicture}[A-node/.style={draw,rectangle,label=70:#1},node distance=22mm,inner sep=5]
\node[A-node={$P$}] at (0,0)(P){a narrative program};
\node[A-node={$a$}] at (5.5,0)(A'){an actant};
\node[A-node={$M$}] at (0,-2.5)(M){a mode};
\node[A-node={$S$}] at (0,2.5)(S){a source text};
\draw[->](P)--node[above]{$\delta$(actorializes)}(A');
\draw[->](P)--node[left]{has junction type}(M);
\draw[<-](S)--node[left]{is discoursivization of}(P);
%\path[->] (P) edge [loop above] node {depends on} ();
\path[->,draw] (P.173) to [out=180, in=180, looseness=7] node [left]{$\epsilon$(depends on)} (P.185);
\end{tikzpicture}
}
\caption{Database table of instances of the object $P \in \mathrm{Ob}(\mathcal{N})$ and olog of the categorical schema $\mathcal{N}$.}
\label{CatN+OlogN}
\end{figure}
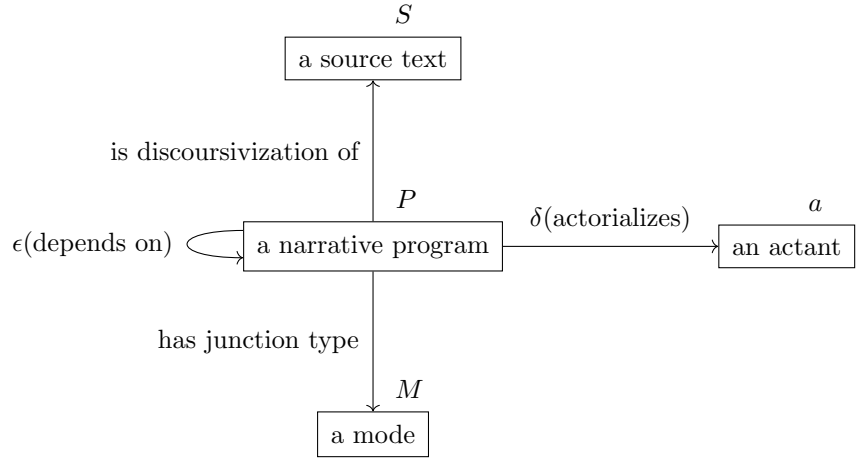
%.....................................................................................................................

Consider the categorical schema $\mathrm{Olog}(\mathcal{N})$ we provide in Figure~\ref{OlogN} that minimally describes the relations between elements found within Greimas's formulation of a NP we gave in Equation~(\ref{NPcap})--(\ref{NPcup}). In the underlying categorical schema $\mathcal{N}$ we utilise the List ($\delta$) and Maybe ($\epsilon$) monads to assign List- and Maybe-valued semantics to the morphisms `actorializes' and `depends on', in the table of $P \in \mathcal{N}$. We retain $\mathcal{N}$ as an ordinary categorical schema, but interpret the aspect (morphism) `actorializes': $P\rightarrow a$ by a List-valued map
\begin{equation}
\delta(\mathrm{actorializes}):I(P)\rightarrow \mathrm{List}(I(a)).
\end{equation}
Thus, `actorializes' is interpreted as a Kleisli arrow in the category $\mathbf{Kls}(\mathrm{List})$, while `depends on' is interpreted as a Kleisli arrow in the category $\mathbf{Kls}(\mathrm{Maybe})$. 

From $\mathrm{List}(I(a))$ we obtain the collection $[x_{1},x_{2}, \dots, x_{n}]$ of role-indexed actants from $I(a)$ that are actorializations of a given NP. The aspect `depends on': $P\rightarrow P$ is defined by a Maybe-value map
\begin{equation}
\epsilon(\mathrm{depends \,\,on}):I(P)\rightarrow \mathrm{Maybe}(I(P)), 
\end{equation}
where $\mathrm{Maybe}(X):=X \cup \{\ast\}$ so that a database entry `$\ast$' denotes the absence of a dependency. Thus for each narrative program $p \in I(P)$, either $\epsilon(\mathrm{depends\,\,on})(p)=p^{\prime}$ for some $p^{\prime}\in P$, or $\epsilon(\mathrm{depends\,\,on})(p)=\ast$, indicating that $p$ is not derived from another NP. We note here that these interpretations involve distinct monads on \textbf{Set}, and hence do not arise from a single Kleisli instance but rather from assigning monadic semantics to individual morphisms. A narrative program can thus be represented minimally as a list-valued Kleisli entry whose remaining elements are actants participating in the junction. Conjunctive or disjunctive modes are carried as row-level metadata rather than encoded internally in the list.

\begin{definition}[List-valued narrative program]\label{LVNPDef}\normalfont
Let $\mathcal{N}$ be the schema of narrative programs such that
\begin{equation}
\delta(\mathrm{actorializes}) : I(P) \to \mathrm{List}(X)
\end{equation}
is the Kleisli interpretation of the morphism labelled as the aspect actorializes: $P \rightarrow a$. A \emph{list-valued narrative program} $p \in I(P)$ is represented by
\begin{equation}
\delta(\mathrm{actorializes})(p)=[x_{1},x_{2}, \dots x_{n}]\in \mathrm{List}(X),
\end{equation}
which we interpret as having the canonical form:
\[
[S_{1},S_{2},X_{1}, \dots, X_{k}]
\]
where $S_{1} \in X$ is the type `a subject of doing,' $S_{2} \in X$ is the type `a subject of state,' and $X_{1}, \dots, X_{k} \in X$ are actants participating in the junction.
\end{definition}

\begin{remark}\label{ObjectRemark}
Although Greimas formulates narrative programs in terms of a relation between a subject and an object in the conjunctive mode as $(S_{2} \cap O)$, and in the disjunctive mode as $(S_{2} \cup O)$, the term \emph{Object} in this context does not correspond strictly to the type $B \in \mathrm{Ob}(\mathcal{A})$ of the actantial model (see Figure~\ref{ActantialModel}). Rather, it denotes a syntactic position, and therefore denotes the entity with which the subject of state is brought into a junction. Due to actantial syncretism, \citep[page 142]{Hebert2020} a given actant may occupy multiple roles within a narrative, or shift roles over time. Consequently, the elements $X_{1}, \dots, X_{k}$ in the list-valued representation of a narrative program are not restricted to instances of the actantial type `Object,' but may include actants of any type (e.g. Helper, Opponent, or Subject). For this reason, we interpret a list-valued narrative program in the form 
\[
[S_1, S_2, X_1, \dots, X_k],
\]
where each $X_i \in X$ is an actant participating in the junction, rather than enforcing a strict typing constraint on the `object' position.
\end{remark}

\begin{definition}[Dependent narrative program]\label{DNPDef}\normalfont
Let $\mathcal{N}$ be the schema of narrative programs and let
\[
\epsilon(\mathrm{depends \,\, on}) : I(P) \to \mathrm{Maybe}(I(P))
\]
be the interpretation of the morphism $\mathrm{dependsOn}:P \to P$. A narrative program $q \in I(P)$ is said to be:
\begin{itemize}
\item \emph{basic} if $\epsilon(\mathrm{dependsOn})(p) = \ast$,
\item \emph{dependent} if $\epsilon(\mathrm{depends \,\,on})(p) = q$ for some $q \in I(P)$.
\end{itemize}
In the dependent case, $p$ is interpreted as extending or composing with the narrative program $q$.
\end{definition}
\begin{remark}
The dependency relation induces a recursive interpretation of narrative programs, whereby the semantics of a dependent program $p$ is obtained by composing its associated transformation with that of the program 
\[
q = \epsilon(\mathrm{depends \,\,on})(p),
\]
when defined. Thus, while a list-valued narrative program records the local actantial configuration of a transformation, the dependency relation provides a means of constructing more complex narrative trajectories through composition.
\end{remark}

If we return to the Aesop fable of \textit{The Hare \& the Tortoise} we can see how the list-valued narrative programs NP1--NP7 compiled in Figure~\ref{TableN} describe the dynamics of the narrative as it unfolds over its four paragraphs. The story begins at timestamp $T_{0}$ where we have the following conjunctive NPs which we parse from the database table of $P \in \mathrm{Ob}(\mathcal{N})$ as:
\begin{align}
NP_{1}:&=\{\textrm{Tortoise} \rightarrow (\textrm{Tortoise} \cap \textrm{challenge})\},\label{NP1Eq}\\
NP_{2}:&=\{\textrm{Hare} \rightarrow (\textrm{Hare} \cap \textrm{underestimation})\}.
\end{align}
These NPs describe the initial challenge of the Tortoise to a race, that itself is based on an underestimation by the Hare of the Tortoise. Here we see what \citet[page 9]{GreimasCourtes1979} identify as the ternary articulation of the 3-tuple (\textit{virtual, actual, realized}), and how it is driven by the mechanism of syncretism. That is to say that the actant `underestimation,' when abstractly considered as a syntagmatic unit is virtualised through the category of an Object, yet when actualised in an NP through a conjunction, may become an Opponent as in $NP_2$. 

In the second paragraph of the story that we timestamp as $T_{1}$, we have the following conjunctive narrative program
\begin{equation}
NP_{3}:=\{\textrm{Fox} \rightarrow (\textrm{Fox} \cap \textrm{fairness})\},
\end{equation}
again with the notion that the actant `fairness' may be virtualised as an Object, but also actualised as a Helper. Then, in the third paragraph at $T_{2}$ we have the following dependent NP which accounts for the modality of \textit{causing-to-be}:
\begin{equation}
NP_{4}:=\{NP_{2}\,\{\textrm{Hare} \rightarrow (\textrm{Hare} \cap \{\textrm{nap, laziness, overconfidence}\})\}\}
\end{equation}
Here we see how the categorical schema $\mathcal{A}$ of Greimas's actantial model provides a typing system, which we use in $\mathcal{N}$ to define a relational structure over those types to produce NPs. That is to say that what  Greimas and Court\'{e}s call \textit{factitiveness}, or how the structure of the factitive modality of causing-to-do or causing-to-be suggests a ``contractual communication (involving the transmission of the modal investment) between two subjects, each of which is endowed with its own narrative trajectory.'' \citep[page 115]{GreimasCourtes1979} This hypotactic relation between Subjects (possibly self-reflective as found at $T_{2}$) is captured through $\epsilon(\mathrm{depends \,\, on})$ in which the semantics of NP4 is not treated as a flat list union, but the composition of the process denoted by NP2 with the process denoted by NP4.

Finally in the last paragraph of the story at $T_{3}$ we have the following narrative programs:
\begin{align}
NP_{5}:&=\{NP_{7}\,\{\textrm{Tortoise} \rightarrow (\textrm{Tortoise} \cap \{\textrm{race win, justification}\})\}\},\\
NP_{6}:&=\{\textrm{Hare} \rightarrow (\textrm{Hare} \cup \textrm{race win})\},\\
NP_{7}:&=\{\textrm{Hare} \rightarrow (\textrm{Hare} \cup \{\textrm{consistency, perserverence}\})\}.\label{NP7Eq}
\end{align}
We interpret the victory of the Tortoise as a dependent NP of the composition: $NP_{5} \circ NP_{7}$. That is to say that the disjunction between the Hare and the actualised Helper actants of `consistency' and `perseverance' are the cause for the realisation of the conjunction between the Tortoise and the Objects of `race win,' and `justification.'

\begin{remark}\label{FGNPRem}%[Faithfulness to Greimasian narrative programs]
The list-valued representation of narrative programs introduced in Definition~\ref{LVNPDef} faithfully captures the structural components of Greimas’s formulation $F\{S_1 \to (S_2 \cap O)\}$ and $F\{S_1 \to (S_2 \cup O)\}$ in the following sense:
\begin{itemize}
\item The distinction between `Subject of doing' and `Subject of state' is preserved explicitly through the first two elements $S_1$ and $S_2$.
\item The junction relation between $S_2$ and associated actants is represented by the list structure $[X_1,\dots,X_k]$, together with the junction type $(\cap$ or $\cup)$ encoded at the schema level.
\item Recursive structure, corresponding to Greimas's use of nested programs via the operator $F$, is captured by the dependency morphism 
$\epsilon(\mathrm{depends \,\,on})$, whose semantics is given by composition (Definition~\ref{DNPDef} and Remark~\ref{ObjectRemark}).
\end{itemize}
Thus, the categorical formulation preserves the essential syntactic and compositional features of Greimasian narrative programs while providing a precise mathematical semantics.
\end{remark}

\section{Diagramming discoursivization}

Although Greimas and others explored the idea of the narrativity in relation to pragmatics, modal analysis and the passions \citep{PerronCollins1989}, the Paris school of structuralist semiotics never formalised a descriptive diagrammatic approach to the theory of how NPs manifest and drive \textit{discoursivization}. Here, we mean what \citet[page 86-86]{GreimasCourtes1979} identify as the procedures of setting a NP into a discourse, and how it must necessarily bridge a gap between ``narrative syntax and narrative semantics.'' In previous work we introduced the use of wiring diagrams in symmetric monoidal categories for narrative analysis that was motivated by a top-down framework for capturing the hypotactical clusters of discoursivization. 

In the second part of this article we take a bottom-up path and directly utilise our established categorical schema $\mathrm{Olog}(\mathcal{N})$ (Figure~\ref{OlogN}) together with our List-valued narrative program constructions (Definition~\ref{LVNPDef}, Definition~\ref{DNPDef}) as the basis for generating a composed wiring diagram of \textit{The Hare \& the Tortoise} that is a function of a \textit{narrative program generator}. We begin by utilising the role-indexed types of the collapsed category $\mathcal{A}^{\prime}$ of the Greimas model as the basis for a \textit{free symmetric monoidal category}.

\begin{definition}[Discrete category of actantial types]\label{DiscreteCategoryDef}\normalfont
Given an actantial instance $I: \mathcal{A}^{\prime}\rightarrow \mathbf{Set}$ with the undelying set $X$, we denote $\mathbf{Disc}(X)$ as the \emph{discrete category} on $X$ whose objects are the element of $X$, and whose only morphisms are the identity morphisms. We interpret $\mathbf{Disc}(X)$ as the category of \emph{actantial types}.
\end{definition}

Deriving the new category $\mathbf{Disc}(X)$ is an analog what Greimas and Court\'{e}s call \textit{actorialization} or the establishing of ``the actors of the discourse by uniting different elements of the semantic and syntactic components.'' \citep[page 7]{GreimasCourtes1979} Here, we directly source the instances (actors) found in the Aesop fable, and establish them as types (objects) in $\mathbf{Disc}(X)$. But in order to allow for the unique generation of NPs from this base, we require the ability for combinatorial possibilities between actants (objects), which we can accommodate through applying a symmetric monoidal structure. 

\begin{definition}[Symmetric monoidal category]\normalfont
A \emph{symmetric monoidal category} is a category $\mathcal{C}$ equipped with:
\begin{itemize}
\item a bifunctor $\otimes : \mathcal{C} \times \mathcal{C} \to \mathcal{C}$, called the \emph{tensor product},
\item a distinguished object $I \in \mathrm{Ob}(\mathcal{C})$, called the \emph{unit object},
\item natural isomorphisms:
\begin{align*}
(X \otimes Y) \otimes Z &\cong X \otimes (Y \otimes Z) && \text{(associativity)}\\
I \otimes X &\cong X \cong X \otimes I && \text{(unit laws)}\\
X \otimes Y &\cong Y \otimes X && \text{(symmetry)}
\end{align*}
\end{itemize}
These isomorphisms satisfy standard coherence conditions (associativity, unit, and symmetry compatibility), ensuring that tensoring is well-behaved up to canonical isomorphism.
\end{definition}

\begin{remark}
Intuitively, the tensor product $\otimes$ represents the co-presence of actants within a narrative configuration, while the symmetry ensures that their ordering is not semantically significant. This becomes relevant at the level of \emph{figurativization} \citep[page 118]{GreimasCourtes1979} or how a \emph{narrative trajectory (NT)} becomes generated firstly through identifying actant subsets or collections that become figurativized through conjunctive $(\cap)$ or disjunctive $(\cup)$ relations.
\end{remark}

\begin{proposition}[Free symmetric monoidal category]\normalfont
Let $I: \mathcal{A}^{\prime}\rightarrow \mathbf{Set}$ be the actantial instance with underlying set $X$. Then there exists a free symmetric monoidal category
\begin{equation}
\mathbf{FSM}(\mathbf{Disc}(X))
\end{equation}
generated by $\mathbf{Disc}(X)$. Its objects are finite tensor products of elements of $X$, including the unit  object $I$, and its morphisms are generated by identities, symmetries, and composition. Moreover, for any symmetric monoidal category $\mathcal{C}$ and any function
\[
f: X \rightarrow \mathrm{Ob}(\mathcal{C}),
\]
there exists a unique (up to isomorphism) strong symmetric monoidal functor
\begin{equation}
\tilde{f}: \mathbf{FSM}(\mathbf{Disc}(X)) \rightarrow \mathcal{C}
\end{equation}
such that $\tilde{f}$ extends $f$.
\end{proposition}

%FIG....................................................................................................................
\begin{figure}
\subfloat[$\mathcal{W}_{1}(\cap_{\mathrm{NP2}},\textrm{causing-to-be};\cap_{\mathrm{NP4}})$.\label{WD-NP4}]{
\resizebox{\textwidth}{!}{
\begin{tikzpicture}[font=\Large,yscale=1.35]
\node[rounded corners=3,draw,minimum size=15mm,inner sep=5] at (0,0) (NP2){$\cap_{\mathrm{NP2}}$};
\node[rounded corners=3,draw,minimum size=15mm,inner sep=5] at (9,0) (ctb){causing-to-be};
\draw(NP2.30)--(NP2.30 -| ctb.west)node[above,pos=0.5]{Hare$_{SubjectState}$};
\draw(NP2.-30)--(NP2.-30 -| ctb.west)node[above,pos=0.5]{underestimation$_{Opponent}$};
\draw(NP2.west)--++(180:10mm)node[left](HareSD){Hare$_{SubjectDoing}$};
\draw(ctb.15)--++(0:10mm)node[right]{{Hare$_{SubjectState}$}};
\draw(ctb.-15) --++(0:10mm)node[right,label={[align=left,yshift=-15,label distance=-5]0:nap$_{Opponent}$$\otimes$\\laziness$_{Opponent}$$\otimes$\\overconfidence$_{Opponent}$}](napOBJ){};
%box
\draw[thick](HareSD.east) [rounded corners=3]--++ (90:12mm)node[above,xshift=180]{$\tau: \left(\bigcap_{\textrm{NP2}},\textrm{causing-to-be}\right)\rightarrow \cap_{\textrm{NP4}}$} -| (napOBJ.west) --++(-90:8mm) -| cycle;
\end{tikzpicture}
}
}

\subfloat[$\mathcal{W}_{2}(\cup_{\mathrm{NP7}},\textrm{causing-to-do};\cap_{\mathrm{NP5}})$.\label{WD-NP5}]{
\resizebox{0.96\textwidth}{!}{
\begin{tikzpicture}[font=\Large,yscale=1.35]
\node[rounded corners=3,draw,minimum size=15mm,inner sep=5] at (0,0) (NP2){$\cup_{\mathrm{NP7}}$};
\node[rounded corners=3,draw,minimum size=20mm,inner sep=5] at (9,0) (ctb){causing-to-do};
\draw(NP2.30)--(NP2.30 -| ctb.west)node[above,pos=0.5]{Hare$_{SubjectState}$};
\draw(NP2.-30)--(NP2.-30 -| ctb.west)node[below,pos=0.5,align=left]{consistency$_{Helper}$$\otimes$\\perseverance$_{Helper}$};
\draw(NP2.west)--++(180:10mm)node[left](HareSD){Hare$_{SubjectDoing}$};
\draw(ctb.25)--++(0:10mm)node[right]{{Tortoise$_{SubjectState}$}};
\draw(ctb.0) --++(0:10mm)node[right,align=left](napOBJ){justification$_{Object}$};
\draw(ctb.-25) --++(0:10mm)node[right,align=left](napOBJ){race win$_{Object}$};
%box
\draw[thick](HareSD.east) [rounded corners=3]--++ (90:12mm)node[above,xshift=180]{$\tau: \left(\cup_{\textrm{NP7}},\textrm{causing-to-do}\right)\rightarrow \bigcap_{\textrm{NP5}}$} -| (napOBJ.west) --++(-90:8mm) -| cycle;
\end{tikzpicture}
}
}
\caption{Wiring diagrams of the narrative programs NP4 and NP5 (cf. Table~\ref{TableN}) from \textit{The Hare \& the Tortoise}.}
\label{WiringDiagrams2}
\end{figure}
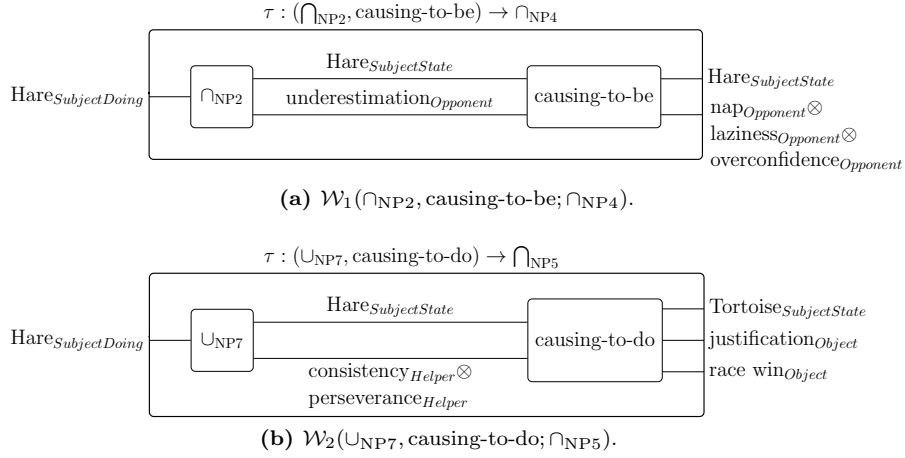
%.....................................................................................................................

Since $\mathbf{Disc}(X)$ contains only identity morphisms, $\mathbf{FSM}(\mathbf{Disc}(X))$ is equivalently the free symmetric monoidal 
category generated by the set $X$. But we want to visualise these types in terms of their actantial associations (figurativization) within a NP, both as inputs (domain side) that are virtualised/actualised in addition to describing them as outputs (codomain side) that produce realisations according to a modality or junction. To address this need we introduce a \textit{narrative program generator} as a non-identity morphism in $\mathbf{FSM}(\mathbf{Disc}(X))$.

\begin{definition}[Narrative program generator]\normalfont
A \emph{narrative program generator} $\llbracket p \rrbracket$ is a morphism in $\mathbf{FSM}(\mathbf{Disc}(X))$ of the form
\begin{equation}
[S_1, S_2, X_1, \dots, X_k] : S_1 \to S_2 \otimes X_1 \otimes \cdots \otimes X_k
\end{equation}
together with a label indicating its conjunctive or disjunctive mode.
\end{definition}

In order to visualise the dynamics of actorialization through the operations of engagement (conjunction) and disengagement (disjunction), we turn here to wiring diagrams in order to provide a visual `map' of the structure of the elements of a NP within a source text.   

\begin{definition}[Wiring diagram]\normalfont
Let $X$ be the set of role-indexed actants obtained from an actantial instance $I : \mathcal{A}^{\prime} \to \mathbf{Set}$, and let $\mathbf{FSM}(\mathbf{Disc}(X))$ be the free symmetric monoidal category generated by $X$. A \emph{wiring diagram} $\mathcal{W}$ is a morphism in 
$\mathbf{FSM}(\mathbf{Disc}(X))$ constructed from:
\begin{itemize}
\item identity morphisms,
\item symmetry isomorphisms,
\item tensor products,
\item composition,
\item and a specified collection of generating morphisms corresponding to 
narrative programs.
\end{itemize}
Graphically, a wiring diagram consists of boxes (representing generators) connected by wires (representing actants), where wires may be rearranged using symmetry and composed sequentially to form larger diagrams.
\end{definition}

We interpret a narrative program generator as a box with one input wire labelled by $S_1$ and multiple output wires labelled by $S_2, X_1, \dots, X_k$ (i.e., an actorialization). Objects of $\mathbf{Disc}(X)$ are interpreted as wire types in a diagrammatic language, and tensor products represent the simultaneous presence of multiple actants. The category $\mathbf{FSM}(\mathbf{Disc}(X))$ provides the typological basis for wiring-diagram representations of narrative. Objects are finite tensor products of role-indexed actants, representing co-presence of actantial elements (figurativization). We give examples of wiring diagrams in Figure~\ref{WiringDiagrams2} that describe what we previously identified in \textit{The Hare \& the Tortoise} as the NPs:
\begin{align*}
NP_{4}:&=\{NP_{2}\,\{\textrm{Hare} \rightarrow (\textrm{Hare} \cap \{\textrm{nap, laziness, overconfidence}\})\}\},\\
NP_{5}:&=\{NP_{7}\,\{\textrm{Tortoise} \rightarrow (\textrm{Tortoise} \cap \{\textrm{race win, justification}\})\}\}.
\end{align*}
Both these NPs correspond to dependent narrative programs, and their dependencies of $NP_2$ and $NP_7$ are shown in the diagrams in Figure~\ref{WiringDiagrams2} through the labelled morphisms (boxes) $\cap_{\mathrm{NP2}}$ and $\cup_{\mathrm{NP7}}$. We follow the convention of \citet{RupelSpivak2013} and name the diagrams in long form as $\mathcal{W}_{n}(Y;Z)$, where $Y$ is the domain as a finite set of boxes (morphisms), and $Z$ is the codomain, or the `exterior' containing box. Each exterior box is also labelled by a function in the form $\phi: Y_{1}, \dots, Y_{n}\rightarrow Z$ that acts as a rule that describes the flow of information in terms of `supply' and `demand.'  We also extend the collection of non-identity morphisms in $\mathbf{FSM}(\mathbf{Disc}(X))$ so as to strengthen the factitive transmission between two utterances of doing.   

\begin{definition}[Factitive morphisms]\normalfont
We introduce two distinguished morphisms in $\mathbf{FSM}(\mathbf{Disc}(X))$:
\begin{align}
\mathrm{ctb} \quad (\text{causing-to-be}), \qquad
\mathrm{ctd} \quad (\text{causing-to-do}),
\end{align}
which correspond to the factitive-modality-NP-construction described by \citet[page 161]{Hebert2020}. These morphisms act on narrative program generators by mediating their composition. Given a narrative program generator 
\begin{equation}
\llbracket p \rrbracket : S_1 \to S_2 \otimes X_1 \otimes \cdots \otimes X_k,
\end{equation}
a dependent narrative program is obtained by composing $\llbracket p \rrbracket$ with a factitive morphism, yielding a new generator
\begin{equation}
\tau : (\llbracket p \rrbracket, \mathrm{ctb}) \longrightarrow \llbracket p \rrbracket^{\prime},
\quad \text{or} \quad
\tau : (\llbracket p \rrbracket, \mathrm{ctd}) \longrightarrow \llbracket p \rrbracket^{\prime},
\end{equation}
where $\llbracket p \rrbracket^{\prime}$ denotes the resulting narrative program.
\end{definition}

These morphisms provide a categorical realization of the nesting operator $F$ in Greimas’s formulation of narrative programs (cf. Equation~(\ref{NPcap})--(\ref{NPcup})). Rather than representing nested programs as lists of lists, we interpret them compositionally: a dependent narrative program is obtained by composing the transformation associated to one program with that of another. The modalities of `causing-to-be' and `causing-to-do' specify the manner in which this composition occurs, in accordance with the semiotic interpretation of factitivity. In diagrammatic terms, these factitive morphisms correspond to the insertion of one wiring diagram into another, as illustrated in Figure~\ref{WiringDiagrams2}. This realizes dependency as substitution of boxes within a larger diagram.

\begin{definition}[Composition via substitution]\label{SubstitutionDef}\normalfont
Let $p, q \in I(P)$ be narrative programs such that $\epsilon(\mathrm{dependsOn})(p) = q$. Let
\begin{equation}
\llbracket p \rrbracket, \llbracket q \rrbracket \in \mathbf{FSM}(\mathbf{Disc}(X))
\end{equation}
denote the corresponding narrative program generators. We define the composition of $\llbracket p \rrbracket$ with $\llbracket q \rrbracket$ as the morphism
\begin{equation}
\llbracket p \rrbracket \circ \llbracket q \rrbracket
\end{equation}
in $\mathbf{FSM}(\mathbf{Disc}(X))$, obtained by substituting the wiring diagram of $\llbracket q \rrbracket$ into the input of $\llbracket p \rrbracket$, whenever the types are compatible.
\end{definition}

In diagrammatic terms, this corresponds to connecting the output wires of  $\llbracket q \rrbracket$ to the input wire of $\llbracket p \rrbracket$, yielding a composite transformation. Thus, narrative programs generate morphisms in a symmetric monoidal category, and narrative structure emerges through their composition via wiring diagrams.

\subsection{Narrative trajectories in a hypergraph category}

Our introduction of morphisms in $\mathbf{FSM}(\mathbf{Disc}(X))$ was motivated by a desire to encode narrative transformations (i.e. through narrative programs) as passages that develop from the static actantial configurations of figurativization to the dynamic narrative processes of actorialization. The passage from $I(a) \in \mathrm{Ob}(\mathcal{N})$ to $\mathbf{FSM}(\mathbf{Disc}(I(a)))$ reinterprets instance-level actants as generating types, enabling their subsequent use as wires in a diagram of the symmetric monoidal semantics of a narrative. 

Our final move is from NPs as local descriptions of the actantial dynamics of syntactic relations expressed through $\llbracket p \rrbracket$ and visualised via $\mathcal{W}$, to what Greimas and Court\'{e}s call a \textit{narrative trajectory} (NT), or ``a hypotactic series of either simple or complex narrative programs . . .  [as] a logical chain in which each NP is presupposed by another, presupposing, NP.'' \citep[page 207]{GreimasCourtes1979}  As simple syntactic units, NPs dictate the ``narrative progress'' of the discourse as well as the dynamics of role-syncretism and the hierarchy of nesting and dependencies. As a collection of interdependent NPs, an NT reinforces the notion that an actant is not a fixed entity but a function of actantial relations that are situated both in the pragmatic and cognitive dimension across the duration of a narrative.  In order to fully explore this within the context of wiring diagrams, we introduce the utility of a \textit{Frobenius structure} and \textit{hypergraph categories}.

%FIG....................................................................................................................
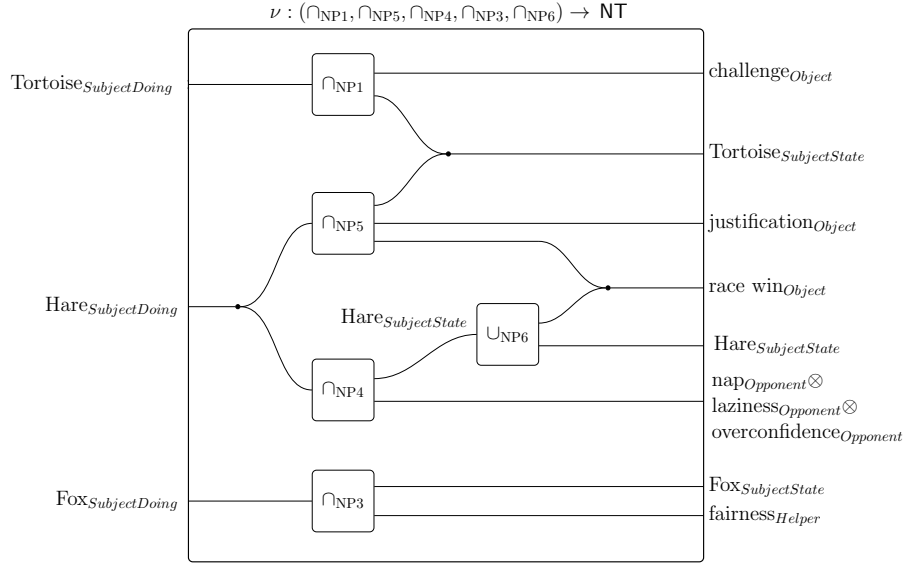
\begin{figure}
\centering
\resizebox{\textwidth}{!}{
\begin{tikzpicture}[font=\Large,yscale=1.35]
\node[rounded corners=3,draw,minimum size=15mm,inner sep=5] at (5,2.5) (NP1){$\cap_{\mathrm{NP1}}$};
\node[rounded corners=3,draw,minimum size=15mm,inner sep=5] at (5,0) (NP5){$\cap_{\mathrm{NP5}}$};
\node[rounded corners=3,draw,minimum size=15mm,inner sep=5] at (5,-3) (NP4){$\cap_{\mathrm{NP4}}$};
\node[rounded corners=3,draw,minimum size=15mm,inner sep=5] at (5,-5) (NP3){$\cap_{\mathrm{NP3}}$};
\node[label={180:Tortoise$_{SubjectDoing}$},left=30mm of NP1]  (TortoiseSD){};
\draw(TortoiseSD.east)--++(0:8mm) coordinate(join1) to [out=0,in=180] (NP1.west);
\node[shape=circle,draw,fill=black,draw,minimum size=1mm,inner sep=0,left=25mm of NP5] at ($(NP5)!0.5!(NP4)$) (join1a){}; 
\draw(join1a) -- (join1a -| TortoiseSD.east)node[label={180:Hare$_{SubjectDoing}$}] (HareSD){};
\draw(join1a) to [out=0,in=180] (NP5.west);
\draw(join1a) to [out=0,in=180] (NP4.west);
\draw(NP3.west) -- (NP3.west -| TortoiseSD.east)node[label={180:Fox$_{SubjectDoing}$}] (FoxSD) (FoxSD){};
\node[rounded corners=3,draw,minimum size=15mm,inner sep=5] at (9,-2) (NP6){$\cup_{\mathrm{NP6}}$};
\draw(NP1.20) --++ (0:80mm)node[right](challengeOBJ){challenge$_{Object}$};
\node[shape=circle,draw,fill=black,draw,minimum size=1mm,inner sep=0,right=25mm of NP5.east] at ($(NP1)!0.5!(NP5)$ |- join1) (join2){}; 
\draw(NP1.-20) to [out=0,in=180] (join2);
\draw(NP5.30) to [out=0,in=180] (join2);
\draw(join2)--(join2 -| challengeOBJ.west)node[right](TortoiseSS){Tortoise$_{SubjectState}$};
\draw(NP5)--(NP5 -| challengeOBJ.west)node[right](justificationOBJ){justification$_{Object}$};
\draw(NP5.-30)--++(0:40mm) coordinate (p1);
\node[shape=circle,draw,fill=black,draw,minimum size=1mm,inner sep=0,right=20mm of p1] at ($(p1)!0.5!(NP6)$ |- p1)(join3){};
\draw(p1) to [out=0,in=180] (join3);
\draw(NP6.20) to [out=0,in=180] (join3);
\draw(join3)--(join3 -| challengeOBJ.west)node[right](racewinOBJ){race win$_{Object}$};
\draw(NP4.20)  to [out=0,in=180] (NP6.180)node[xshift=-50,above]{Hare$_{SubjectState}$};
\draw(NP4.-20) -- (NP4.-20 -| TortoiseSS.west)node[right,label={[align=left,yshift=-5,label distance=-5]0:nap$_{Opponent}$$\otimes$\\laziness$_{Opponent}$$\otimes$\\overconfidence$_{Opponent}$}] (napOPP) (napOPP){};
\draw(NP6.-20) -- (NP6.-20 -| TortoiseSS.west)node[label={[align=left]0:Hare$_{SubjectState}$}] (HareSS) (napOPP){};
\draw(NP3.25)--(NP3.25 -| challengeOBJ.west)node[right](FoxSS){Fox$_{SubjectState}$};
\draw(NP3.-25)--(NP3.-25 -| challengeOBJ.west)node[right](fairnessOBJ){fairness$_{Helper}$};
%box
\draw[thick](TortoiseSD.east) [rounded corners=3]--++ (90:10mm)node[above,xshift=180]{$\nu: \left(\cap_{\textrm{NP1}},\cap_{\textrm{NP5}},\cap_{\textrm{NP4}},\cap_{\textrm{NP3}},\cap_{\textrm{NP6}}\right) \rightarrow\, \mathsf{NT}$} -| (challengeOBJ.west) --++(-90:88mm) -| cycle;
\end{tikzpicture}
}
\caption{The wiring diagram $\mathcal{W}_{0}(\cap_{\textrm{NP1}},\cap_{\textrm{NP5}},\cap_{\textrm{NP4}},\cap_{\textrm{NP3}},\cap_{\textrm{NP6}};\mathsf{NT})$ of the narrative trajectory of the Aesop fable \textit{The Hare \& Tortoise}.}
\label{WiringDiagram0}
\end{figure}
%.....................................................................................................................

\begin{definition}[Special commutative Frobenius structure]\normalfont
Let $(\mathcal{C}, \otimes, I)$ be a symmetric monoidal category. An object $X \in \mathcal{C}$
is said to carry a special commutative Frobenius structure if it is equipped with morphisms:
\begin{equation}
\begin{tikzpicture}
\tikzset{
Aspect/.style={line width=0.75},
B-node/.style={shape=circle,draw,fill=black,draw,minimum size=1mm,inner sep=0}
}
%mu
\begin{scope}
\node[B-node] at (0,0)(dot1){};
\draw[Aspect] (dot1)--++(180:5mm);
\draw[Aspect] (0.5,0.25) to[out=180,in=90](dot1);
\draw[Aspect] (0.5,-0.25) to[out=180,in=-90](dot1);
\node[below=5mm of dot1]{$\delta_{X}: X \rightarrow X \otimes X$};
\end{scope}
\begin{scope}[xshift=33mm]
\node[B-node] at (0,0)(dot2){};
\draw[Aspect] (dot2)--++(180:5mm);
\node[below=5mm of dot2]{$\epsilon_{X} : X\rightarrow I$};
\end{scope}
\begin{scope}[xshift=63mm]
\node[B-node] at (0,0)(dot3){};
\draw[Aspect] (dot3)--++(0:5mm);
\draw[Aspect] (-0.5,0.25) to[out=0,in=90](dot3);
\draw[Aspect] (-0.5,-0.25) to[out=0,in=-90](dot3);
\node[below=5mm of dot3]{$\mu_{X} : X \otimes X \rightarrow X$};
\end{scope}
\begin{scope}[xshift=93mm]
\node[B-node] at (0,0)(dot2){};
\draw[Aspect] (dot2)--++(0:5mm);
\node[below=5mm of dot2]{$\eta_{X} : I\rightarrow X$};
\end{scope}
\end{tikzpicture}
\end{equation}
such that $(X, \eta_{X},\mu_{X})$ is a commutative monoid, and $(X,\delta_{X},\epsilon_{X})$ is a cocommutative comonoid that satisfy the Frobenius law:
\begin{equation}
(\mu_X \otimes \mathrm{id}) \circ (\mathrm{id} \otimes \delta_X)=\delta_X \circ \mu_X=(\mathrm{id} \otimes \mu_X) \circ (\delta_X \otimes \mathrm{id}),
\end{equation}
and specialness axiom:
\begin{equation}
\mu_X \circ \delta_X = \mathrm{id}_X.
\end{equation}
\end{definition}

\begin{remark}
We note here that the Frobenius structure enables operations that are essential for modeling narrative flow: actants may be duplicated, merged, or discarded as they participate in multiple narrative programs. This reflects the semiotic phenomenon of actantial persistence and transformation across a narrative trajectory.
\end{remark}

\begin{definition}[Hypergraph category]\normalfont
A hypergraph category is a symmetric monoidal category $(\mathcal{C}, \otimes, I)$ such that every object $X \in \mathcal{C}$
is equipped with the structure of a special commutative Frobenius monoid, that is, with morphisms
\begin{equation}
\delta : X \to X \otimes X, \quad
\epsilon : X \to I, \quad
\mu : X \otimes X \to X, \quad
\eta : I \to X,
\end{equation}
satisfying the associativity, unitality, commutativity, Frobenius, and specialness axioms. These Frobenius structures are required to be compatible with the symmetric monoidal product, in the sense that for all objects $X,Y \in \mathcal{C}$, the induced Frobenius structure on the tensor product $X \otimes Y$ agrees with the tensor product of the Frobenius structures on $X$ and $Y$. 
\end{definition}

Consider the wiring diagram $\mathcal{W}_{0}$ in the hypergraph category $\mathcal{H}(X)$ of the narrative trajectory of the Aesop fable \textit{The Hare \& the Tortoise} we give in Figure~\ref{WiringDiagram0}. Here, the class of `Subject of doing' ($S_{1}$) actants supply $\mathsf{NT}$ as the actor triple
 \[
 (\mathrm{Tortoise}_{SubjectDoing}\otimes \mathrm{Hare}_{SubjectDoing}\otimes \mathrm{Fox}_{SubjectDoing}),
 \]
with the codomain of $\mathsf{NT}$ a collection of corresponding `Subject of state' ($S_{2}$) and actant junctures (Objects, Opponents, Helpers) that can be individually traced to their $\llbracket \mathrm{NP} \rrbracket$-domains. Through the substitution rule we gave in Definition~\ref{SubstitutionDef}, we see that the diagrams of $\mathcal{W}_{1}$ (Figure~\ref{WD-NP4}) and $\mathcal{W}_{2}$ (Figure~\ref{WD-NP5}) can be substituted with the morphisms $\cap_{\mathrm{NP4}} \in \mathrm{Ob}(\mathcal{W}_{0})$ and $\cap_{\mathrm{NP5}}\in \mathrm{Ob}(\mathcal{W}_{0})$ respectively. Composition via substitution is implemented diagrammatically using both sequential composition and the Frobenius structure, which together allow wiring diagrams to be connected even when interfaces require duplication, merging, or reordering.

More specifically, the substitutions in $\mathcal{W}_{0}$ track the factitive modalities of the Hare's underestimation of the Tortoise. This causes the Hare to be lazy, overconfident and asleep for part of the race (causing-to-be). It also tracks the Hare's lack of consistency and perseverance which in turn allows the Tortoise to justify his initial confidence and claim the race win (causing-to-do). In $\mathcal{W}_{0}$ we also see the impartiality of the Fox as Helper expressed through its actant-isolated $\llbracket \mathrm{NP3} \rrbracket$, while the Frobenius morphisms of 
\[
\delta(\mathrm{Hare}_{SubjectDoing}),\quad \mu(\mathrm{Tortoise}_{SubjectState}),\quad \mu(\mathrm{race \,\,win}_{Object}),
\]
track the syntactic engagement between the Hare and the Tortoise and how particular actants express intersections of complementarity (e.g., the Object `race win').

\begin{definition}[Narrative trajectory]\normalfont
Let $\mathcal{H}(X)$ denote the hypergraph category generated from $\mathbf{FSM}(\mathbf{Disc}(X))$ by equipping each object with a 
special commutative Frobenius structure. A \emph{narrative trajectory} ($\mathsf{NT}$) is a morphism
\begin{equation}
\nu : A \to B
\end{equation}
in $\mathcal{H}(X)$ obtained from a finite family of narrative program generators by composition, tensor product, symmetry, and Frobenius operations. Equivalently, an $\mathsf{NT}$ is a composed wiring diagram whose boxes are narrative program generators and whose wires are role-indexed actants.
\end{definition}

\begin{remark}
The presence of a special communicative Frobenius structure in $\mathcal{H}(X)$, allows for our representation of the narrative trajectory in the Aesop fable to be expressed in $\mathcal{W}_{0}$ via \emph{copy} $(\delta_{X})$ and \emph{join} $(\mu_{X})$ distinctions on types. At the semiotic level, these operations accommodate what Greimas and Court\'{e}s identify as the multiple roles that an actant may assume in a narrative given an actant is ``defined both by the position of the actant in the logical sequence of the narration (its syntactic definition) and its modal investment (its morphological definition).'' \citep[page 6]{GreimasCourtes1979} Thus, in a hypergraph category we can show the multiplicity of roles for an actant and track their connection to particular NP inputs (as virtualisations/actualisations) or outputs (realisations) and therefore the flow and dynamics of their modal engagement or disengagement in a discourse. A narrative trajectory is thus a hypergraph morphism generated by narrative programs.
\end{remark}

\begin{theorem}[Compositional realization of narrative trajectory]\label{theoremNT}\normalfont
Let $X$ be the set of role-indexed actants obtained from the actantial instance $I : \mathcal{A}' \to \mathbf{Set}$ associated to the Aesop fable \textit{The Hare \& the Tortoise}, and let $\mathcal{H}(X)$ be the hypergraph category generated from $\mathbf{FSM}(\mathbf{Disc}(X))$. Let $p_1, \dots, p_7 \in I(P)$ denote the list-valued narrative programs defined in Equations~(\ref{NP1Eq})--(\ref{NP7Eq}), together with the dependency relation $\epsilon(\mathrm{dependsOn})$. Then there exists a morphism
\[
\nu : \bigotimes_{i \in \{1,3,4,5,6\}} \llbracket p_i \rrbracket \;\longrightarrow\; \mathsf{NT}
\]
in $\mathcal{H}(X)$, constructed by composition, tensor product, symmetry, 
and Frobenius operations, such that:
\begin{itemize}
\item $\nu$ composes the generators corresponding to $p_1,\dots,p_7$ 
according to the dependency structure induced by 
$\epsilon(\mathrm{dependsOn})$,
\item dependent narrative programs (e.g. $p_4$, $p_5$) are obtained by 
substitution of wiring diagrams corresponding to their dependencies (i.e., $p_{2}$, $p_{7}$),
\item factitive morphisms (causing-to-be, causing-to-do) mediate the 
composition of nested programs,
\item Frobenius structure enables duplication, merging, and routing of 
actants across the diagram.
\end{itemize}
Moreover, $\nu$ is represented diagrammatically by the wiring diagram 
\[\mathcal{W}_{0}(\cap_{\textrm{NP1}},\cap_{\textrm{NP5}},\cap_{\textrm{NP4}},\cap_{\textrm{NP3}},\cap_{\textrm{NP6}};\mathsf{NT})
\]
given in Figure~\ref{WiringDiagram0}.
\end{theorem}

\begin{proof}[Proof sketch]
Each narrative program $p_i$ determines a generator $\llbracket p_i \rrbracket$ in $\mathcal{H}(X)$. The dependency relation 
$\epsilon(\mathrm{dependsOn})$ induces a partial order on these generators, which determines the order of composition. Dependent programs are incorporated via substitution of wiring diagrams, realized as composition in $\mathcal{H}(X)$. Factitive morphisms specify 
the mode of this composition, while Frobenius operations provide the necessary structural maps to duplicate and merge actants. The resulting composite morphism $\nu$ is therefore well-defined and is represented by the wiring diagram $\mathcal{W}_{0}$ in Figure~\ref{WiringDiagram0}.
\end{proof}

%FIG....................................................................................................................
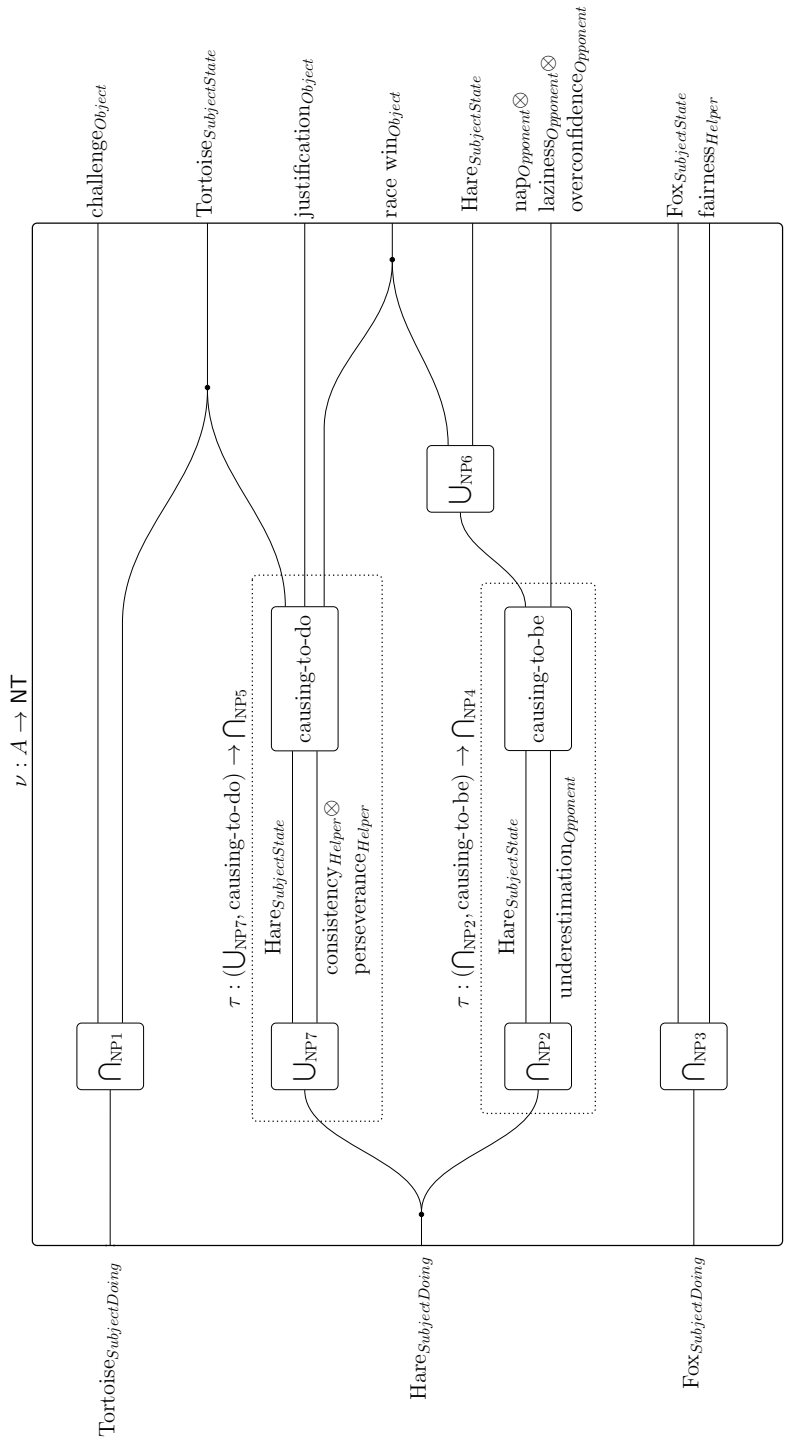
\begin{sidewaysfigure}
\centering
\resizebox{\textwidth}{!}{
\begin{tikzpicture}[font=\Large,yscale=1.75]
\node[rounded corners=3,draw,minimum size=15mm,inner sep=5] at (5,2.5) (NP1){$\bigcap_{\mathrm{NP1}}$};
\node[rounded corners=3,draw,minimum size=15mm,inner sep=5] at (5,0) (NP5){$\bigcup_{\mathrm{NP7}}$};
\node[rounded corners=3,draw,minimum size=15mm,inner sep=5] at (5,-3) (NP4){$\bigcap_{\mathrm{NP2}}$};
\node[rounded corners=3,draw,minimum size=15mm,inner sep=5] at (5,-5) (NP3){$\bigcap_{\mathrm{NP3}}$};
\node[label={180:Tortoise$_{SubjectDoing}$},left=35mm of NP1]  (TortoiseSD){};
\draw(TortoiseSD.east)--++(0:8mm) coordinate(join1) to [out=0,in=180] (NP1.west);
\node[shape=circle,draw,fill=black,draw,minimum size=1mm,inner sep=0,left=35mm of NP5] at ($(NP5)!0.5!(NP4)$) (join1a){}; 
\draw(join1a) -- (join1a -| TortoiseSD.east)node[label={180:Hare$_{SubjectDoing}$}] (HareSD){};
\draw(join1a) to [out=0,in=180] (NP5.west);
\draw(join1a) to [out=0,in=180] (NP4.west);
\draw(NP3.west) -- (NP3.west -| TortoiseSD.east)node[label={180:Fox$_{SubjectDoing}$}] (FoxSD) (FoxSD){};
\node[rounded corners=3,draw,minimum size=15mm,inner sep=5] at (18,-2) (NP6){$\bigcup_{\mathrm{NP6}}$};
\draw(NP1.20) --++ (0:180mm)node[right](challengeOBJ){challenge$_{Object}$};
\node[shape=circle,draw,fill=black,draw,minimum size=1mm,inner sep=0,right=150mm of NP5.east] at ($(NP1)!0.5!(NP5)$ |- join1) (join2){}; 
\draw(NP1.-20) --++(0:90mm) to [out=0,in=180] (join2);
\draw(join2)--(join2 -| challengeOBJ.west)node[right](TortoiseSS){Tortoise$_{SubjectState}$};
\node[shape=circle,draw,fill=black,draw,minimum size=1mm,inner sep=0,right=90mm of p1] at ($(p1)!0.5!(NP6)$ |- p1)(join3){};
\draw(NP6.20) to [out=0,in=180] (join3);
\draw(join3)--(join3 -| challengeOBJ.west)node[right](racewinOBJ){race win$_{Object}$};
\draw(NP6.-20) -- (NP6.-20 -| TortoiseSS.west)node[label={[align=left]0:Hare$_{SubjectState}$}] (HareSS) (napOPP){};
\draw(NP3.25)--(NP3.25 -| challengeOBJ.west)node[right](FoxSS){Fox$_{SubjectState}$};
\draw(NP3.-25)--(NP3.-25 -| challengeOBJ.west)node[right](fairnessOBJ){fairness$_{Helper}$};
\node[rounded corners=3,draw,minimum size=15mm,inner sep=5] at (13.5,-3) (ctb){causing-to-be};
\draw(NP4.20) -- (NP4.20 -| ctb.west)node[pos=0.5,above]{Hare$_{SubjectState}$};
\draw(NP4.-20) -- (NP4.-20 -| ctb.west)node[pos=0.5,below]{underestimation$_{Opponent}$};
\draw(ctb.-10) -- (ctb.-10 -| TortoiseSS.west)node[label={[align=left]0:nap$_{Opponent}$$\otimes$\\laziness$_{Opponent}$$\otimes$\\overconfidence$_{Opponent}$}] (napOPP) (napOPP){};
\draw(ctb.10) to [out=0,in=180] (NP6);
\node [draw, rectangle, fit=(NP4) (ctb), inner sep=15pt,rounded corners=3, dotted,thick,label={90:$\tau: \left(\bigcap_{\textrm{NP2}},\textrm{causing-to-be}\right)\rightarrow \bigcap_{\textrm{NP4}}$}] {};
\node[rounded corners=3,draw,minimum size=15mm,inner sep=5] at (13.5,0) (ctd){causing-to-do};
\draw(ctd.-15)--++(0:40mm) to [out=0,in=180] (join3);
\draw(ctd.0)--(ctd.0 -| challengeOBJ.west)node[right](justificationOBJ){justification$_{Object}$};
\draw(NP5.-20) -- (NP5.-20 -| ctd.west)node[below,pos=0.5,align=left]{consistency$_{Helper}$$\otimes$\\perseverance$_{Helper}$};
\draw(ctd.15) to [out=0,in=180] (join2);
\draw(NP5.20) -- (NP5.20 -| ctd.west)node[pos=0.5,above]{Hare$_{SubjectState}$};
\node [draw, rectangle, fit=(NP5) (ctd), inner sep=20pt,rounded corners=3, dotted,thick,label={90:$\tau: \left(\bigcup_{\textrm{NP7}},\textrm{causing-to-do}\right)\rightarrow \bigcap_{\textrm{NP5}}$},yshift=-8] {};
%box
\draw[thick](TortoiseSD.east) [rounded corners=3]--++ (90:10mm)node[above,xshift=330]{$\nu : A \to \mathsf{NT}$} -| (challengeOBJ.west) --++(-90:88mm) -| cycle;
\end{tikzpicture}
}
\caption{The expanded wiring diagram of $\nu : A \to \mathsf{NT}$ that describes the narrative trajectory of the Aesop fable \textit{The Hare \& Tortoise} after substitution of $\llbracket NP4 \rrbracket$ and $\llbracket NP5 \rrbracket$.}
\label{WiringDiagramExpanded}
\end{sidewaysfigure}
%.....................................................................................................................

\begin{remark}
The wiring diagram in Figure~\ref{WiringDiagram0} presents the narrative trajectory in a factorized form, in which dependent narrative programs (e.g. $\llbracket \mathrm{NP4}\rrbracket$ and $\llbracket \mathrm{NP5}\rrbracket$) are represented as single generators. Figure~\ref{WiringDiagramExpanded} provides an expanded representation of the same morphism, in which these generators are replaced by their corresponding wiring diagrams (cf. Figure~\ref{WiringDiagrams2}). This distinction mirrors the difference between a syntactic description 
of a process and its fully expanded operational form. The dotted regions indicate the sites of substitution. Hence, these two diagrams represent the same morphism 
\[
\nu \in \mathcal{H}(X),
\]
related by substitution of generators, rather than by additional composition.
\end{remark}

\section{Discussion}

In Theorem~\ref{theoremNT} we establish that the narrative trajectory of the Aesop fable is not a sequence of isolated transformations, but a single compositional morphism whose structure is determined by the interaction of narrative programs. Thus, a narrative trajectory is realized as a hypergraph morphism generated by narrative programs. This consequently allows us to view the morphism $\nu$ as a utility for what Parret identifies as the purpose of semiotic investigation as a means to ``reveal a typology of competent subjects with their specific modal trajectories.'' \citep[page 112]{Parret1983} Indeed Greimas and Court\'{e}s argue that a narrative trajectory, composed of multiple narrative programs, may itself be regarded as an actant---what they term a \emph{functional} or \emph{syntagmatic} actant:
\begin{quote}
The narrative trajectory contains as many actantial roles as there are NPs constituting it. Consequently the set of actantial roles of a narrative trajectory may be called [an] actant or, in order to distinguish it from the NP's syntactic actants, [a] functional (or syntagmatic) actant. So defined, the actant is not a concept which is fixed once and for all, but is a virtuality subsuming an entire narrative trajectory. \citep[page 207]{GreimasCourtes1979}
\end{quote}
In our framework, this observation admits a natural categorical interpretation. A narrative trajectory is represented as a morphism
\[
\nu : A \to B
\]
in the hypergraph category $\mathcal{H}(X)$, constructed from narrative program generators $\llbracket p_{i} \rrbracket$ via composition, tensor product, and Frobenius operations. The collection of actants participating across the trajectory corresponds to the family of wires that are composed, duplicated, and merged within this morphism. 

From this perspective, the trajectory $\nu$ itself determines a higher-order actantial unit: not a single element of $X$, but a structured configuration of actants given by the interface $(A,B)$ together with the internal wiring of $\nu$. In particular, $\nu$ may be regarded as an actant in a higher-level category of processes, where morphisms themselves become the carriers of actantial roles. This formalizes Greimas's notion of a functional actant as a \emph{virtual}  entity: one that does not correspond to a fixed instance, but to a compositional structure subsuming an entire narrative trajectory.

It follows then that narrative trajectories exhibit a form of categorical lifting, whereby actants at the level of objects give rise to actants at the level of morphisms. Diagrammatically, this corresponds to interpreting an entire wiring diagram as a single unit, whose external interface defines its actantial role, while its internal structure encodes the interactions of its constituent actants. Thus, the functional actant is realized as a compositional diagram rather than as an individual entity.

\section{Conclusion}

In this article we introduced Greimasian narrative programs and proposed a categorical formalization that makes their compositional structure explicit. Building on the reconstruction of the actantial model as the categorical schema $\mathcal{A}$, we introduced Kleisli semantics over the List and Maybe monads to represent narrative programs as non-atomic, list-valued structures together with a dependency relation encoding their recursive composition. This provided a minimal yet expressive framework in which the internal structure of narrative programs is captured at the level of data.

To account for the dynamics of discoursivization, we then lifted these constructions into a diagrammatic setting by freely generating a symmetric monoidal, and subsequently hypergraph category $\mathcal{H}(X)$ from a discrete category of role-indexed actants. In this setting, narrative programs act as generators of morphisms, and their composition is realized through wiring diagrams. The introduction of Frobenius structure allowed for the duplication, merging, and routing of actants, reflecting the persistence and transformation of roles across a narrative.

Within this framework, a narrative trajectory is understood as a single compositional morphism constructed from interacting narrative programs. This provides a precise mathematical interpretation of Greimas's claim that narratives are formed through the combination of elementary syntactic units, while also extending his insight that a trajectory itself may be regarded as a functional or syntagmatic actant. Here, such an actant is realized not as an individual entity, but as a structured process encoded by a wiring diagram.

More broadly, this work suggests that structural semiotics admits a natural interpretation in terms of compositional systems, in which meaning emerges through the interaction and transformation of relational units. We anticipate that future work may explore extensions of this approach to richer modal structures, multi-agent narratives, and computational implementations, as well as further connections to operadic and higher-categorical formulations of narrative composition. By bringing together categorical schemas, Kleisli semantics, and diagrammatic reasoning, we provide a unified framework that connects the static representation of narrative elements with their dynamic realization in discourse.

From our perspective, the Greimasian project may be reinterpreted as an early intuition of compositionality. That is, meaning is not located in isolated units, but arises through structured relations and their transformations. The present framework makes this intuition precise, showing that narrative programs and their trajectories can be understood as morphisms in a compositional category. Narrative meaning is thus not merely represented, but constructed as a compositional process.

%%%%%%%%%%%%%%%%%%%%%%%%%%%%%%%%%%%%%%%%%%%%%%%%%%%%%%%%%%%%%%%%%
\bibliographystyle{authordate1}
\bibliography{CategoryTheoretic.bib}
%%%%%%%%%%%%%%%%%%%%%%%%%%%%%%%%%%%%%%%%%%%%%%%%%%%%%%%%%%%%%%%%%
\clearpage
\appendix

\section{Appendix}
\subsection*{The Hare \& the Tortoise}\label{Hare+Tortoise}

A Hare was making fun of the Tortoise one day for being so slow.
``Do you ever get anywhere?'' he asked with a mocking laugh.
``Yes,'' replied the Tortoise, ``and I get there sooner than you think. I'll run you a race and prove it.''\\

\noindent The Hare was much amused at the idea of running a race with the Tortoise, but for the fun of the thing he agreed. So the Fox, who had consented to act as judge, marked the distance and started the runners off.\\

\noindent The Hare was soon far out of sight, and to make the Tortoise feel very deeply how ridiculous it was for him to try a race with a Hare, he lay down beside the course to take a nap until the Tortoise should catch up.\\

\noindent The Tortoise meanwhile kept going slowly but steadily, and, after a time, passed the place where the Hare was sleeping. But the Hare slept on very peacefully; and when at last he did wake up, the Tortoise was near the goal. The Hare now ran his swiftest, but he could not overtake the Tortoise in time.\\

\noindent \textit{The race is not always to the swift.} \citep{Aesop1884}
\end{document}